\documentclass[twoside]{amsart}
\usepackage{latexsym}
\usepackage{amssymb,amsmath,amsopn}
\usepackage[dvips]{graphicx}   
\usepackage{color,epsfig}      
\usepackage{tikz} 
\usepackage{subcaption} 
\usepackage{array}
\usepackage{eurosym}
\usepackage{hyperref} 
\newtheorem{thm}{Theorem}

\newtheorem{lem}{Lemma}

\theoremstyle{definition}

\newtheorem{rem}{Remark}

\newtheorem{conj}{Conjecture}
\newtheorem{prob}{Problem}

\renewcommand{\Re}{\mathbb R}

\DeclareMathOperator{\Id}{Id}

\DeclareMathOperator{\conv}{conv}

\DeclareMathOperator{\surf}{surf}

\DeclareMathOperator{\vol}{vol}

\begin{document}

\title[An isoperimetric problem]{An isoperimetric problem for three-dimensional parallelohedra}

\author[Z. L\'angi]{Zsolt L\'angi}

\thanks{The author is supported by the National Research, Development and Innovation Office, NKFI, K-134199, the J\'anos Bolyai Research Scholarship of the Hungarian Academy of Sciences, and grants BME IE-VIZ TKP2020 and \'UNKP-20-5 New National Excellence Program by the Ministry of Innovation and Technology.}

\address{Morphodynamics Research Group and Department of Geometry, Budapest University of Technology, Egry J\'ozsef utca 1., Budapest 1111, Hungary}  \email{zlangi@math.bme.hu}

\keywords{parallelohedra, zonotopes, discrete isoperimetric problems, tiling, Kepler's conjecture, honeycomb conjecture, Kelvin's conjecture}

\subjclass[2010]{52B60, 52A40, 52C22}

\begin{abstract}
The aim of this note is to investigate isoperimetric-type problems for $3$-dimensional parallelohedra; that is, for convex polyhedra whose translates tile the $3$-dimensional Euclidean space. Our main result states that among $3$-dimensional parallelohedra with unit volume the one with minimal mean width is the regular truncated octahedron.
\end{abstract}

\maketitle

\section{Introduction}

Among the convex polyhedra in Euclidean $3$-space $\Re^3$ most known both inside and outside mathematics, we find the so-called \emph{parallelohedra}; that is, the convex polyhedra whose translates tile space. All parallelohedra in $3$-space can be defined as the Minkowski sums of at most six segments with prescribed linear dependencies between the generating segments, and thus, they are a subclass of \emph{zonotopes}. Parallelohedra in $\Re^3$ are also related to the Voronoi cells of lattices via Voronoi's conjecture, stating (and proved in $\Re^3$, see \cite{Delone, Fedorov, Garber}) that any parallelohedron is an affine image of such a cell. Well-known examples of parallelohedra are the cube, the regular rhombic dodecahedron, and the regular truncated octahedron, which are the Voronoi cells generated by the integer, the face-centered cubic and the body-centered cubic lattice, respectively.

Many papers investigate isoperimetric problems for zonotopes (see, e.g. \cite{Bezdek1, Bezdek2, Linhart}), with special attention on the relation between mean width and volume (see \cite{Daroczy,Filliman}). On the other hand, apart from the celebrated proof of Kepler's Conjecture by Hales \cite{Hales2}, which, as a byproduct, implies that among parallelohedra with a given inradius, the one with minimum volume is the regular rhombic dodecahedron, there is no known isoperimetric-type result for parallelohedra. 

In our paper we offer a new representation of $3$-dimensional parallelohedra that might be useful to investigate their geometric properties. We use this representation to prove Theorem~\ref{thm:3}.

\begin{thm}\label{thm:3}
Among unit volume $3$-dimensional parallelohedra, regular truncated octahedra have minimal mean width.
\end{thm}

The structure of the paper is as follows. In Section~\ref{sec:prelim} we introduce and parametrize $3$-dimensional parallelohedra, and describe the tools and the notation necessary to prove our result. In Section~\ref{sec:thm3} we prove Theorem \ref{thm:3}. Finally, in Section~\ref{sec:remarks}, we collect additional remarks, and recall some old, and raise some new problems. 

\section{Preliminaries}\label{sec:prelim}

\subsection{Properties of three-dimensional parallelohedra}\label{subsec:par_prop}

Minkowski \cite{Minkowski} and Ven\-kov \cite{Venkov} proved that $P$ is a convex $n$-dimensional polytope for which there is a face-to-face tiling of $\Re^n$ with translates of $P$ if and only if
\begin{itemize}
\item[(i)] $P$ and all its facets are centrally symmetric, and
\item[(ii)] the projection of $P$ along any of its $(n-2)$-dimensional faces is a parallelogram or a centrally symmetric hexagon.
\end{itemize}
Dolbilin \cite{Dolbilin} showed that these properties are already necessary if we require only the existence of a not necessarily face-to-face tiling with translates of $P$. The polytopes satisfying properties (i-ii) are called \emph{$n$-dimensional parallelohedra}.

The combinatorial classes of $3$-dimensional parallelohedra are well known (cf. Figure~\ref{fig:types}). More specifically, any $3$-dimensional parallelohedron is combinatorially isomorphic to one of the following:
\begin{enumerate}
\item[(1)] a cube,
\item[(2)] a hexagonal prism,
\item[(3)] Kepler's rhombic dodecahedron (which we call a regular rhombic dodecahedron),
\item[(4)] an elongated rhombic dodecahedron,
\item[(5)] a regular truncated octahedron.
\end{enumerate}
We call a parallelohedron combinatorially isomorphic to the polyhedron in (i) a \emph{type (i) parallelohedron}.
It is also known that every parallelohedron $P$ in $\Re^3$ is a zonotope. More specifically,
a type (1)-(5) parallelohedron can be attained as the Minkowski sum of $3,4,4,5,6$ segments, respectively. A typical parallelohedron is of type (5), as every other parallelohedron can be obtained by removing some of the generating vectors of a type (5) parallelohedron.

Since every $3$-dimensional parallelohedron $P \subset \Re^3$ is a zonotope, every edge of $P$ is a translate of one of the segments generating $P$. Furthermore, for any generating segment $S$, the faces of $P$ that contain a translate of $S$ form a \emph{zone}, i.e. they can be arranged in a sequence $F_1, F_2, \ldots, F_k, F_{k+1}=F_1$ of faces of $P$ such that for all values of $i$, $F_i \cap F_{i+1}$ is a translate of $S$. By property (ii) in the list in the beginning of Section~\ref{sec:prelim}, this sequence contains $4$ or $6$ faces. In these cases we say that $S$ generates a $4$-belt or a $6$-belt in $P$, respectively.
The numbers of $4$- and $6$-belts of a type (i) parallelohedron are $3$ and $0$, $3$ and $1$, $0$ and $4$, $1$ and $4$, and $0$ and $6$ for $i=1,2,3,4,5$, respectively.

\begin{figure}[ht]
\begin{center}
\includegraphics[width=0.75\textwidth]{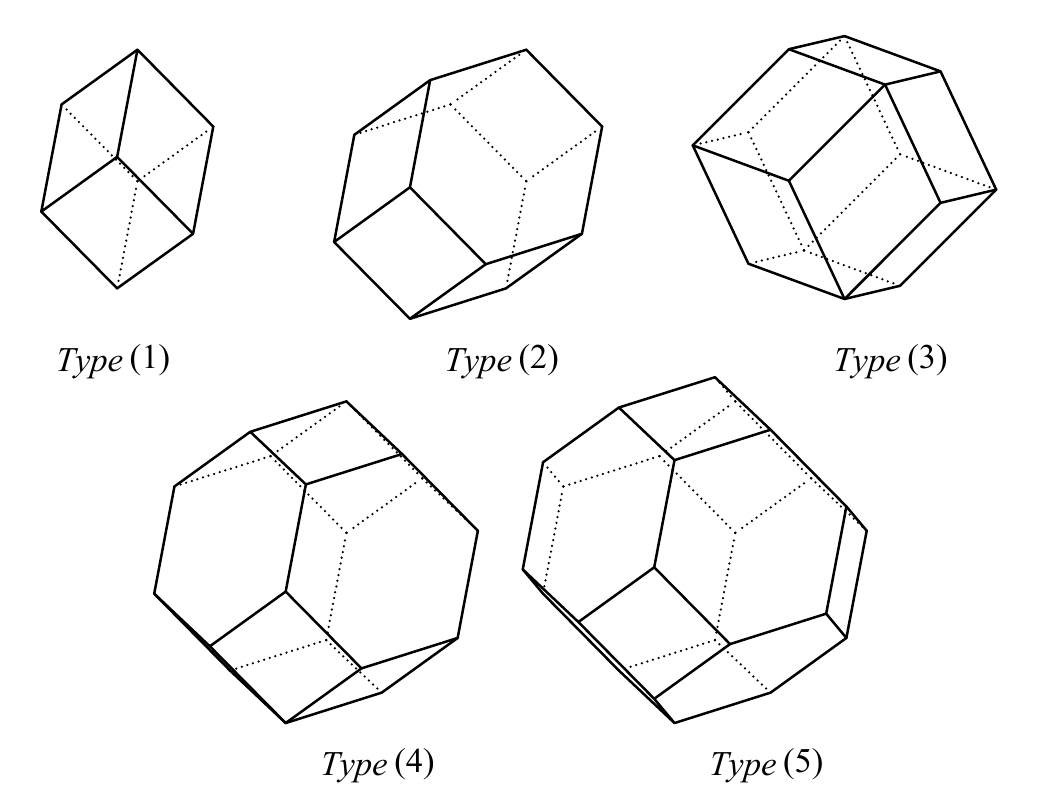}
\caption{The five combinatorial types of $3$-dimensional parallelohedra. The type (5) parallelohedron in the picture is the regular truncated octahedron generated by the six segments connecting the midpoints of opposite edges of a cube. The type (3) polyhedron is the regular rhombic dodecahedron generated by the four diagonals of the same cube. The rest of the polyhedra in the picture are obtained by removing some generating segments from the type (5) polyhedron.}
\label{fig:types}
\end{center}
\end{figure}


\begin{rem}\label{rem:quermassintegrals}
Let $P=\sum_{i=1}^m [o,v_i] \subset \Re^3$ be a parallelohedron, and set $I= \{ 1,2,\ldots, m\}$. Then the volume, the surface area and the mean width of $P$ is
\begin{equation}\label{eq:vol_par}
\vol_3(P)= \sum_{\{ i,j,k \} \subset I} |V_{ijk}|,
\end{equation}
\begin{equation}\label{eq:surf_par}
\surf(P)  = 2 \sum_{\{ i,j \} \subset I} | v_i \times v_j|, \hbox{ and}
\end{equation}
\begin{equation}\label{eq:width_par}
w(P) = \frac{1}{2} \sum_{i=1}^m |v_i|,
\end{equation}
respectively, where $V_{ijk}$ is the determinant of the matrix with column vectors $v_i,v_j,v_k$.
\end{rem}

\begin{proof}
The formula in (\ref{eq:vol_par}) is the well-known volume formula for zonotopes (cf. e.g. \cite{Shephard}).
The one in (\ref{eq:surf_par}) can be proved directly for type (5) parallelohedra, which implies its validity for all parallelohedra.
Finally, Steiner's formula \cite{Schneider} yields that the second quermassintegral $W_2(P)$ of $P$ is $\frac{1}{6}\sum_{ij} (\pi-\alpha_{ij}) l_{ij}$, where $\alpha_{ij}$ and $l_{ij}$ is the dihedral angle and the length of the edge between the $i$th and $j$th faces of $P$. Thus, from the zone property it follows that $W_2(P) = \frac{\pi}{3} \sum_{i=1}^m |v_i|$. On the other hand, for any convex body $K \subset \Re^3$ we have $w(K) = \frac{3}{2\pi} W_2(K)$, which implies (\ref{eq:width_par}).
\end{proof}

If we choose the generating segments in the form $[o,p_i]$ for some vectors $p_i \in \Re^3$, then 
in case of a type (1) or type (4) parallelohedron, the vectors $v_i$ are in general position, i.e. any three of them are linearly independent, whereas in case of the other types there are triples of vectors that are restricted to be co-planar. In particular, if $P$ is a type (5) parallelohedron, then each of the four pairs of hexagon faces defines a linear dependence relation on the generating segments. More specifically, if $v_1, v_2, v_3, v_4$ are normal vectors of the four pairs of hexagon faces of the parallelohedron, then the directions of the six generating segments are the ones defined by the six cross-products $v_i \times v_j$, $i \neq j$. Note that all triples of the $v_i$s are linearly independent, as otherwise some of the generating segments of $P$ are collinear. Thus, we have
\begin{equation}\label{eq:representation}
P = \sum_{1 \leq i < j \leq 4} [o, \beta_{ij}(v_i \times v_j)]
\end{equation}
for some real numbers $\beta_{ij}$, which, without loss of generality, we may assume to be non-negative.
Every $3$-dimensional parallelohedron can be obtained from a type (5) parallelohedron by removing some of the generating segments, or equivalently, by setting some of the $\beta_{ij}$s to be zero. In particular, $P$ is of type (5), or (4), or (3), or (2), or (1) if and only if
\begin{itemize}
\item $\beta_{ij} > 0$ for all $i \neq j$, or
\item exactly one of the $\beta_{ij}$s is zero, or
\item exactly two of the $\beta_{ij}$s is zero: $\beta_{i_1j_1} = \beta_{i_2,j_2} = 0$, and they satisfy $\{ i_1, j_1 \} \cap \{ i_2, j_2 \} = \emptyset$, or
\item exactly two of the $\beta_{ij}$s is zero: $\beta_{i_1j_1} = \beta_{i_2,j_2} = 0$, and they satisfy $\{ i_1, j_1 \} \cap \{ i_2, j_2 \} \neq \emptyset$, or
\item exactly three of the $\beta_{ij}$s are zero. and there is no $s \in \{1,2,3,4\}$ such that the indices of all nonzero $\beta_{ij}$s contain $s$, respectively.
\end{itemize}
It is easy to check that in the remaining cases $P$ is planar.

Thus, we can use this representation for all parallelohedra appearing in the paper with the assumption that $\beta_{ij} \geq 0$ for all $1 \leq i < j \leq 4$. 

Since no three of the vectors $v_i$ are linearly independent, but clearly all four of them are, up to multiplying by a constant they have a unique nontrivial linear combination $\lambda_1 v_i + \lambda_2 v_2 + \lambda_3 v_3 + \lambda_4 v_4= o$. Since in the representation of $P$ we considered only the directions of the $v_i$s, we may clearly assume that $v_1+v_2+v_3+v_4 = o$. This implies that the tetrahedron $\conv \{ v_1, v_2, v_3, v_4 \}$ is \emph{centered}, that is, its center of gravity is $o$. From the equation $v_1+v_2+v_3+v_4 = o$ it follows that the absolute values of the determinants of any three of the $v_i$s are equal. We choose this common value $1$, which yields that the volume of the tetrahedron is $\frac{2}{3}$.
Throughout the paper for any $i,j,k \in \{ 1,2,3,4 \}$, we denote by $V_{ijk}$ the value of the determinant with $v_i,v_j,v_k$ as column vectors. We choose the indices in such a way that $V_{123}=1$. Since $v_1+v_2+v_3+v_4 = o$, we have that for any $\{i,j,s,t\} = \{ 1,2,3,4\}$, the plane containing $o, v_i, v_j$ strictly separates $v_s$ and $v_t$, which implies that $V_{ijs}=-V_{ijt}$.

In the proof we often use the function $f : \Re^6 \to \Re$,
\begin{equation}\label{eq:volume}
f(\tau_{12},\tau_{13},\tau_{14},\tau_{23},\tau_{24},\tau_{34}) = \tau_{12} \tau_{13} \tau_{23} + \tau_{12}\tau_{14}\tau_{24} + \tau_{13}\tau_{14}\tau_{34} + \tau_{23} \tau_{24} \tau_{34} +
\end{equation}
\[
+ (\tau_{12}+\tau_{34})(\tau_{13}\tau_{24}+\tau_{14}\tau_{23})+ (\tau_{13}+\tau_{24})(\tau_{12}\tau_{34}+\tau_{14}\tau_{23}) + (\tau_{14}+\tau_{23})(\tau_{12}\tau_{34}+\tau_{13}\tau_{24}).
\]
We remark that $f$ is \emph{not} the third elementary symmetric function on six variables, since after expansion it has only $16 < 20 = \binom{6}{3}$ members.

An elementary computation yields that for any $i,j,k,l,s,t \in \{1,2,3,4\}$, we have
\[
|v_i \times v_j, v_k \times v_l, v_s \times v_t| = V_{ijl} V_{stk} - V_{ijk} V_{stl}.
\]
Using this formula and the properties of cross-products, we may express the volume, surface area and mean width of $P$ in our representation.

\begin{rem}\label{rem:quermass_rep}
For the parallelohedron $P=\sum_{i \neq j} \beta_{ij} [o,v_i \times v_j]$ satisfying the conditions above, we have
\begin{equation}\label{eq:vol_rep}
\vol_3(P)= f(\beta_{12},\beta_{13},\beta_{14},\beta_{23},\beta_{24},\beta_{34}),
\end{equation}
\begin{equation}\label{eq:surf_rep}
\surf(P)  = 2 \big( (\beta_{12}\beta_{13}+\beta_{12}\beta_{14}+\beta_{13} \beta_{14}) |v_1| + \ldots (\beta_{14}\beta_{24}+\beta_{14}\beta_{34}+\beta_{24} \beta_{34}) |v_4| +
\end{equation}
\[
+ \beta_{12}\beta_{34} (|v_1+v_2|)+ \ldots + \beta_{14}\beta_{23} |v_1+v_4| \big),
\]
\begin{equation}\label{eq:width_rep}
w(P) = \frac{1}{2} \sum_{1 \leq i < j \leq 4} \beta_{ij} |v_i \times v_j|,
\end{equation}
respectively.
\end{rem}

\subsection{Preliminary lemmas}

The following lemma is used more than once in the paper.

\begin{lem}\label{lem:tetrahedron_identities}
Let $T=\conv \{ p_1,p_2,p_3,p_4\}$ be an arbitrary centered tetrahedron with volume $V>0$ where the vertices are labelled in such a way that the determinant with columns $p_1,p_2,p_3$ is positive.
For any $\{i,j,s,t\} = \{1,2,3,4\}$, let $\gamma_{ij}=- \langle p_s, p_t \rangle$
and $\zeta_{ij}= \gamma_{ij} | p_i \times p_j|^2$. Then
\begin{enumerate}
\item $f(\gamma_{12},\gamma_{13},\gamma_{14},\gamma_{23},\gamma_{24},\gamma_{34}) = \frac{9}{4}V^2$,
\item $\sum_{1 \leq  i < j \leq 4} \zeta_{ij} = \frac{27}{4}V^2$,
\end{enumerate}
where $f$ is the function defined in (\ref{eq:volume}).
\end{lem}
 
\begin{proof}
Let $\chi_{ij} = \langle p_i, p_j \rangle$ for all $i,j$. Consider the Gram matrix $G$ defined by $p_1,p_2,p_3$, and observe that as $T$ is centered, the volume of the parallelepiped spanned by them is $\frac{3}{2}V$. Since the determinant of the Gram matrix of $n$ linearly independent vectors in $\Re^n$ is the square of the volume of the parallelotope spanned by the vectors, it follows that $\det(G) = \frac{9}{4}V^2$.
Furthermore, since $T$ is centered, we have $\chi_{14}=-\chi_{11}-\chi_{12}-\chi_{13}$, and we may 
obtain similar formulas for $\chi_{24}$ and $\chi_{34}$. Now, substituting these formulas into $f(\gamma_{12},\ldots,\gamma_{34})=f(-\chi_{34},\ldots,-\chi_{12})$, we obtain that
\[
f(\gamma_{12},\ldots,\gamma_{34}) = \chi_{11} \chi_{22} \chi_{33} + 2 \chi_{12} \chi_{13} \chi_{23} - \chi_{11} \chi_{2}^2 - \chi_{22} \chi_{13}^2 - \chi_{33} \chi_{12}^2 = \det(G),
\]
which implies the first identity.

To obtain the second identity, observe that for any $\{i,j,s,t \} = \{ 1,2,3,4\}$, we have $\zeta_{ij} = - \langle p_s, p_t \rangle | p_i \times p_j|^2 = - \chi_{st} \left( \chi_{ii} \chi_{jj} - \chi_{ij}^2 \right)$. This observation and an argument similar to the one in the previous paragraph yields that $\sum_{1 \leq  i < j \leq 4} \zeta_{ij} = 3 \det(G) = \frac{27}{4}V^2$.
\end{proof}

\begin{lem}\label{lem:computations}
Under the condition that $\sum_{1 \leq i < j \leq 4} \tau_{ij} = C > 0$ and $\tau_{ij} \geq 0$ for all $1 \leq i < j \leq 4$, we have
\[
f(\tau_{12},\tau_{13},\tau_{14},\tau_{23},\tau_{24},\tau_{34}) \leq \frac{2C^3}{27},
\]
with equality if and only if $\tau_{ij}= \frac{C}{6}$ for all $1 \leq i < j \leq 4$.
Furthermore, if, in addition, $\tau_{34} = 0$, then
\[
f(\tau_{12},\tau_{13},\tau_{14},\tau_{23},\tau_{24},0) \leq \frac{16C^3}{243},
\]
with equality if and only if $2 \tau_{12} = \tau_{13} = \ldots = \tau_{24} = \frac{2C}{9}$.
\end{lem}

\begin{proof}
Without loss of generality, we assume that $C=1$.
Since the set of points in $\Re^6$ satisfying the conditions of Lemma~\ref{lem:computations} is compact, $f$ attains its maximum on it.
Assume that $f$ is maximal at some point $\tau=(\tau_{12}, \ldots, \tau_{34})$.

\emph{Case 1}, $\tau$ has only positive coordinates. By the Lagrange multiplier method, the gradient of $f$ at $\tau$ is is parallel to the gradient of the function $\sum_{i \neq j} \tau_{ij}-1$ defining the condition. In other words, all partial derivatives of $f$ are equal.
Let us assume that this common value is some $t \in \Re$.
Since $f$ is a linear function of any of its variables, $\partial_{\tau_{ij}} f$ is equal to the coefficient of $\tau_{ij}$ in $f$. On the other hand, any member of the sum in $f$ contributes to partial derivatives with respect to three of the variables. Thus, we have that
$t = \sum_{i \neq j} t \tau_{ij} = \sum_{i \neq j} (\partial_{\tau_{ij}} f)(\tau) \tau_{ij} = 3 f(\tau)$.

Computing the partial derivatives, we have that
\[
(\partial_{\tau_{12}} f) (\tau)= \tau_{13} \tau_{23} + \tau_{14} \tau_{24} + \tau_{13} \tau_{24} + \tau_{14} \tau_{23} + (\tau_{13}+\tau_{24} + \tau_{14}+\tau_{23}) \tau_{34},
\]
and we obtain similar expressions for all partial derivatives. Solving the system of equations
$(\partial_{\tau_{1i}} f)(\tau) = 3 f(\tau)$, $i =2,3,4$ for $\tau_{12}, \tau_{13}, \tau_{14}$, and substituting back the solutions into the definition of $f$, we obtain that $f(\tau) = \frac{4}{27}(\tau_{23}+\tau_{24}+\tau_{34})$.
By a similar argument,
$f(\tau)=\frac{4}{27}(\tau_{12}+\tau_{13}+\tau_{23}) = \frac{4}{27}(\tau_{12}+\tau_{14}+\tau_{24}) = \frac{4}{27}(\tau_{13}+\tau_{14}+\tau_{34})$ also follows.
Adding up, we obtain that in this case $4f(\tau) = \frac{8}{27} \sum_{i \neq j} \tau_{ij} = \frac{8}{27}$, implying that $f(\tau)= \frac{2}{27}$.

We show that in case of equality all coordinates of $\tau$ are equal.
The equations in the previous paragraph yield that $\tau_{12}+\tau_{13}+\tau_{23} = \ldots = \tau_{23}+\tau_{24}+\tau_{34} = \frac{1}{2}$, and also that
$\tau_{12}+\tau_{13}+\tau_{14} = \ldots = \tau_{14}+\tau_{24}+\tau_{34} = \frac{1}{2}$.
Comparing suitable equations it readily follows that $\tau_{12}=\tau_{34}$, $\tau_{13}=\tau_{24}$ and $\tau_{14}=\tau_{23}$.
Replacing $\tau_{34}$, $\tau_{24}$ and $\tau_{23}$ by $\tau_{12}, \tau_{13}$ and $\tau_{14}$, respectively, in the equations $(\partial_{\tau_{12}} f) (\tau)= (\partial_{\tau_{13}} f)(\tau) = (\partial_{\tau_{14}} f)(\tau)$, we obtain $\tau_{12}^2 = \tau_{13}^2 = \tau_{14}^2$, which characterizes equality in this case.

\emph{Case 2}, if exactly one coordinate of $\tau$ is zero. Without loss of generality, we may assume that $\tau_{34} = 0$, and let $\tau' \in \Re^5$ be the vector obtained by removing the last coordinate of $\tau$. 
Let us define the function $g(\tau_{12}, \ldots, \tau_{24})= f(\tau_{12}, \ldots, \tau_{24},0)$. Similarly like in Case 1, if $f(\tau)$ is a local maximum, we have that
all five partial derivatives of $g$ are equal at $\tau'$. Again, from this it follows that this common value is equal to $3g(\tau')$. Factorizing the left-hand sides of the equations $\partial_{\tau_{13}}-\partial_{\tau_{14}}=0$ and $(\partial_{\tau_{23}}g)(\tau')-(\partial_{\tau_{24}}g)(\tau')=0$, we obtain that $\tau_{13}=\tau_{14}$, $\tau_{23}=\tau_{24}$, respectively. After substituting these equalities into the equation $(\partial_{\tau_{13}}g)(\tau')-(\partial_{\tau_{23}}g)(\tau')=0$ and factorizing the left-hand side, we obtain that $\tau_{13}=\tau_{23}$. Again, substituting back into $(\partial_{\tau_{12}}g)(\tau')-(\partial_{\tau_{13}} g)(\tau')=0$ and factorizing its left-hand side yields $\tau_{13}= 2 \tau_{12}$. From this, we have $\tau_{12}= \frac{1}{9}$ and $\tau_{ij}=\frac{2}{9}$ for any $i \in \{1,2\}$ and $j \in \{ 3,4\}$, which implies that $g(\tau')=  \frac{16}{243}$.

\emph{Case 3}, if at least two coordinates of $\tau$ are zero. Then, using an argument similar to that in the previous cases, we obtain that
the maximum of $f(\tau)$ is $\frac{1}{16}$, $\frac{1}{27}$ and $0$ if exactly two, exactly three or more than three coordinates of $\tau$ are zero, respectively.
\end{proof}

\section{Proof of Theorem~\ref{thm:3}}\label{sec:thm3}

Recall from Subsection~\ref{subsec:par_prop} that we represent $P$ in the form
\[
P= \sum_{1 \leq i < j \leq 4} [o,\beta_{ij} v_i \times v_j]
\]
for some  $v_1, v_2, v_3, v_4 \in \Re^3$ satisfying $\sum_{i=1}^4 v_i = o$, and for some $\beta_{ij} \geq 0$. By our assumptions, $|V_{ijk}| = 1$ for any $\{i,j,k \} \subset \{ 1,2,3,4\}$, where $V_{ijk}$ denotes the determinant with $v_i,v_j,v_k$ as columns, and we assumed that $V_{123}=1$, which implies, in particular, that $V_{124}=-1$.
Then $T=\conv \{ v_1, v_2, v_3, v_4 \}$ is a centered tetrahedron, with volume $\vol_3(T)=\frac{2}{3}$.

By Remark~\ref{rem:quermass_rep}, we need to find the minimum value of $w(P) = \frac{1}{2} \sum_{1 \leq i < j \leq 4} \beta_{ij} |v_i \times v_j|$
under the conditions that $T= \conv \{ v_1, v_2, v_3, v_4 \}$ is a centered tetrahedron with volume $\vol_3(T)=\frac{2}{3}$, and the value of $f(\beta_{12},\ldots,\beta_{34})$ being fixed. Equivalently, we need to find the maximum value of $f(\beta_{12},\ldots,\beta_{34})$ under the condition that $w(P) = \frac{1}{2} \sum_{1 \leq i < j \leq 4} \beta_{ij} |v_i \times v_j|$ is fixed, where $T= \conv \{ v_1, v_2, v_3, v_4 \}$ is an arbitrarily chosen centered tetrahedron with volume $\vol_3(T)=\frac{2}{3}$.

Our main goal is to reduce this optimization problem to the special case where $T$ is a regular tetrahedron, and then apply Lemma~\ref{lem:computations}.
We do it in two steps. In the first step we find an upper bound for $f$ depending only on $T$ (and thus, we eliminate all $\beta_{ij}$s from the conditions). In the second step we show that our upper bound from the first step can be bounded from above by a value of $f$ under the additional condition that $T$ is regular (or in other words, we eliminate the tetrahedron $T$ from the conditions at the price of bringing back the parameters $\beta_{ij}$).

\emph{Step 1}.\\
Our main tool in this step is Lemma~\ref{lem:isotropic}, proved by Petty \cite{Petty} and later rediscovered by Giannopoulos and Papadimitrakis \cite{Giannopoulos} for arbitrary convex bodies (see (3.4) in \cite{Giannopoulos}). Before stating it, we remark that any affine image of a parallelohedron is also a parallelohedron.

\begin{lem}\label{lem:isotropic}
Let $P \subset \Re^n$ be a convex polytope with outer unit facet normals $u_1, \ldots, u_k$. Let $F_i$ denote the $(n-1)$-dimensional volume of the $i$th facet of $P$. Then, up to congruence, there is a unique volume preserving affine transformation $L$ such that $\surf(L(P))$ is maximal in the affine class of $P$. Furthermore, $P$ satisfies this property if and only if its surface area measure is isotropic, that is, if
\begin{equation}\label{eq:isotropic}
\sum_{i=1}^k \frac{n F_i}{\surf(P)} u_i \otimes u_i = \Id
\end{equation}
where $\Id$ denotes the identity matrix.
\end{lem}

Any convex polytope satisfying the conditions in (\ref{eq:isotropic}) is said to be in \emph{surface isotropic position}.
We note that the volume of the projection body of any convex polyhedron is invariant under volume preserving linear transformations (cf. \cite{Petty2}).
On the other hand, from Cauchy's projection formula and the additivity of the support function (see. \cite{Gardner}) it follows that the projection body of the polytope in Lemma~\ref{lem:isotropic} is the zonotope $\sum_{i_1}^k [o,F_i u_i]$. Note that the solution to Minkowski's problem \cite{Schneider} states that some unit vectors $u_1, \ldots, u_k \in \Re^n$ and positive numbers $F_1, \ldots, F_k$ are the outer unit normals and volumes of the facets of a convex polytope if and only if the $u_i$s span $\Re^n$, and $\sum_{i=1}^k F_i u_i = o$. On the other hand, a translate of the parallelohedron $P$ in our investigation can be written in the form $\frac{1}{2}\sum_{1 \leq i < j \leq 4} [o, \pm \beta_{ij} v_i \times v_j]$, which, by the previous observation, is the projection body of a centrally symmetric polytope with outer unit facet normals $\pm \frac{v_i \times v_j}{|v_i \times v_j|}$, and volumes of the corresponding facets $\frac{\beta_{ij}}{2} |v_i \times v_j|$, where $1 \leq i < j \leq 4$.
Since $u \otimes u = (-u) \otimes (-u)$ for any $u \in \Re^n$, any solution $P$ to our optimization problem satisfies the conditions in (\ref{eq:isotropic}) with the vectors $\frac{v_i \times v_j}{|v_i \times v_j|}$ in place of the $u_i$s, and the quantities $\beta_{ij}  |v_i \times v_j|$ in place of the $F_i$s.
Thus, applying also the formula in Remark~\ref{rem:quermass_rep} for $w(P)$, it follows that
\begin{equation}\label{eq:isotropic_w}
\sum_{1 \leq i < j \leq 4} \frac{3 \beta_{ij}}{2w(P) |v_i \times v_j|} (v_i \times v_j) \otimes (v_i \times v_j) = \Id,
\end{equation}
which in the following we assume to hold for $P$.

Recall that $x \otimes y = x y^T$ and $\langle x,y \rangle = x^T y$ for any column vectors $x,y \in \Re^n$.
Hence, multiplying both sides in (\ref{eq:isotropic_w}) by $v_3^T$ from the left and $v_4$ from the right, it follows that
\[
\langle v_3, v_4 \rangle = \frac{3 \beta_{12}}{2 w(P) |v_1 \times v_2|} V_{123} V_{124}.
\]
Since $V_{123} = - V_{124} = 1$, we can express $\beta_{12}$ as $\beta_{12} = - \langle v_3, v_4 \rangle |v_1 \times v_2| \frac{2w(P)}{3}$. By symmetry, for any $\{i,j,s,t\} = \{ 1,2,3,4\}$ it follows that $\beta_{ij} = - \langle v_s, v_t \rangle |v_i \times v_j| \frac{2 w(P)}{3}$, where it may be worth noting that as $\beta_{ij} \geq 0$ holds for any value of $i,j$, if $P$ satisfies (\ref{eq:isotropic_w}), then $\langle v_s, v_t \rangle \leq 0$ for all $s,t$.
Set $\zeta_{ij} = - \langle v_s, v_t \rangle |v_i \times v_j| = \frac{3 \beta_{ij}}{2w(P)} $ for all $i,j$.
Substituting back these quantities into the formulas for $\vol_3(P)$ and $w(P)$ and simplifying, we may rewrite our optimization problem in the following form:
find the maximum of $\frac{27 \vol_3(P)}{\left(2 w(P) \right)^3} = f(\zeta_{12}, \ldots, \zeta_{34})$, where $\zeta_{ij}$ is defined as above, in the family of all centered tetrahedra $T$ with $\vol_3(T)=\frac{2}{3}$, under the condition that $\sum_{1 \leq i < j \leq 4} \zeta_{ij} |v_i \times v_j| = 3$. Here, it is worth noting that the last condition is satisfied for any centered tetrahedron $T$ with volume $\frac{2}{3}$ by Lemma~\ref{lem:tetrahedron_identities}, and thus, it is redundant.

\emph{Step 2}.\\
Consider the function
\[
f(\zeta_{12}, \ldots, \zeta_{34}), \hbox{ where } \zeta_{ij}= - \langle v_s, v_t \rangle |v_i \times v_j| \hbox{ for } \{i,j,s,t \} = \{ 1,2,3,4\},
\]
and $v_1,v_2,v_3,v_4$ are the vertices of a centered tetrahedron $T$ with volume $\frac{2}{3}$.
Set $\gamma_{ij} = - \langle v_s, v_t \rangle$ and $\tau_{ij}= \gamma_{ij} |v_i \times v_j|^2$ for all $\{i,j,s,t\} = \{ 1,2,3,4\}$.
To give an upper bound on the value of $f$, we apply the Cauchy-Schwartz Inequality, which states that for any nonnegative real numbers $x_i,y_i$, $i=1,2,\ldots, k$, we have
$\sum_{i=1}^k x_i y_i \leq \sqrt{\sum_{i=1}^k x_i^2} \sqrt{\sum_{i=1}^k y_i^2}$, with equality if and only if the $x_i$s and the $y_i$s are proportional.
To do this, we write each member of $f$ as the product $\zeta_{ij} \zeta_{kl} \zeta_{mn} = \sqrt{\gamma_{ij} \gamma_{kl} \gamma_{mn}} \sqrt{\tau_{ij} \tau_{kl} \tau_{mn}}$. Thus, we obtain that
\[
f(\zeta_{12}, \ldots, \zeta_{34}) \leq \sqrt{f(\gamma_{12}, \ldots, \gamma_{34})} \sqrt{f(\tau_{12}, \ldots, \tau_{34})} = \sqrt{f(\tau_{12}, \ldots, \tau_{34})},
\]
where we used the fact that by Lemma~\ref{lem:tetrahedron_identities}, $f(\gamma_{12}, \ldots, \gamma_{34}) = 1$ for all centered tetrahedra with volume $\frac{2}{3}$.
Furthermore, observe that by Lemma~\ref{lem:tetrahedron_identities} we have $\sum_{1 \leq i \leq j \leq 4} \tau_{ij} = 3$ for all such tetrahedra. Hence, 
by Lemma~\ref{lem:computations}, we have $f(\zeta_{12}, \ldots, \zeta_{34}) \leq \sqrt{2}$.

Now, assume that $f(\zeta_{12}, \ldots, \zeta_{34}) = \sqrt{2}$. Then, by Lemma~\ref{lem:computations}, $\tau_{ij} = \frac{1}{2}$ for all values $i \neq j$. This, by the Cauchy-Schwartz Inequality, implies that for some $t \in \Re$, $\gamma_{ij} = t$ for all values $i \neq j$. Since $T$ is centered, this implies that $\gamma_{ii} = 3t$ for all $i$s.
In other words, the Gram matrix of the vectors $v_1,\ldots,v_4$ is a scalar multiple of the matrix $4 \Id - E$, where $E$ is the matrix with all entries equal to $1$.
Since the Gram matrix of a vector system determines the vectors up to orthogonal transformations, and it is easy to check that $4 \Id - E$ is the Gram matrix of the vertex set of a centered regular tetrahedron of circumradius $\sqrt{3}$, the equality part in Theorem~\ref{thm:3} follows.

\section{Remarks and open problems}\label{sec:remarks}

First, in the next table we collect the results of our investigation for the minimal values of the mean widths of the different types of parallelohedra.

Recall (cf. Section~\ref{sec:prelim}) that the different types of parallelohedra correspond to the following polyhedra.

\begin{itemize}
\item Type (1) parallelohedra: parallelopipeds.
\item Type (2) parallelohedra: hexagonal prisms.
\item Type (3) parallelohedra: rhombic dodecahedra.
\item Type (4) parallelohedra: elongated rhombic dodecahedra.
\item Type (5) parallelohedra: truncated octahedra.
\end{itemize}

\begin{table}[h]
\begin{tabular}{>{\centering}m{0.1\textwidth}||>{\centering}m{0.25\textwidth}|>{\centering}m{0.5\textwidth}}
Type & Minimum of $w(P)$ & Optimal parallelohedra\tabularnewline
\hline \hline
(1) & $\frac{3}{2}$ & cube\tabularnewline
\hline
(2) & $\frac{3^{7/6}}{2^{4/3}} \approx 1.43$ & regular hexagon based right prism, with base and lateral edges of lengths $\frac{2^{2/3}}{3^{5/6}}$ and $\frac{3^{1/6}}{2^{1/3}}$, respectively \tabularnewline
\hline
(3) & $\frac{3^{1/2}}{2^{1/3}} \approx 1.37$ & regular rhombic dodecahedron with edge length $\frac{\sqrt{3}}{2^{4/3}}$\tabularnewline
\hline
(4) & $\geq \frac{3^{4/3}}{2^{5/3}} \approx 1.36$ & not known \tabularnewline
\hline
(5) & $\frac{3}{2^{7/6}} \approx 1.34$ & regular truncated octahedron of edge length $\frac{1}{2^{7/6}}$
\end{tabular}
\caption{Minima of the mean widths of different types of unit volume parallelohedra}
\end{table}

In the remaining part of this section we collect some old conjectures about $3$-dimensional parallelohedra, and propose some new ones.

The origin of the Honeycomb Conjecture, stating that in a decomposition of the Euclidean plane into regions of equal area, the regular hexagonal grid has the least perimeter, can be traced back to ancient times \cite{Varro}. This problem has been in the focus of research throughout the second half of the 20th century \cite{LFT, Morgan}, and was finally settled by Hales \cite{Hales1}.

The most famous analogous conjecture for mosaics in $3$-dimensional Euclidean space is due to Lord Kelvin \cite{Kelvin}, who in 1887 conjectured that in a tiling of space with cells of unit volume, the mosaic with minimal surface area is composed of slightly modified truncated octahedra. Even though this conjecture was disproved by Weaire and Phelan in 1994 \cite{Weaire}, who discovered a tiling of space with two slighly curved `polyhedra' of equal volume and with less total surface area than in Lord Kelvin's mosaic, the original problem of finding the mosaics with equal volume cells and minimal surface area has been extensively studied (see, e.g. \cite{Kusner, Gabbrielli, Oudet}). On the other hand, in the author's knowledge, there is no subfamily of mosaics for which Kelvin's problem is solved. 
This is our motivation for the following conjecture, where we note that a $3$-dimensional parallelohedron of unit volume has minimal surface area if and only if the same holds for the translative, convex mosaic generated by it.

\begin{conj}
Among $3$-dimensional, unit volume parallelohedra $P$, $\surf(P)$ is minimal if and only if $P$ is a regular truncated octahedron.
\end{conj}

The following conjecture, called Rhombic Dodecahedral Conjecture, can be found in \cite{Bezdek1} (see also \cite{Bezdek2}). We note that this conjecture is the lattice variant of the so-called Strong Dodecahedral Conjecture \cite{Bezdek3} proved by Hales in \cite{Hales3}. For other variants of the Dodecahedral Conjecture, see also \cite{BezdekLangi} or \cite{Hales4}.

\begin{conj}[Bezdek, 2000]
The surface area of any parallelohedron in $\Re^3$ with unit inradius is at least as large as $12 \sqrt{2} \approx 16.97$, which is the surface area of the regular rhombic dodecahedron of unit inradius. 
\end{conj}

As it was mentioned in the introduction, Voronoi cells of lattice packings of congruent balls is an important subfamily of $3$-dimensional parallelohedra.
Regarding this subclass, Dar\'oczy \cite{Daroczy} gave an example of a packing whose Voronoi cells have a smaller mean width than that of the regular rhombic dodecahedron of the same inradius. This is the motivation behind our last question. Here we note that the mean width of a cube, a regular truncated octahedron and a regular rhombic dodecahedron of unit inradius is $3 > \sqrt{6} = \sqrt{6}$, respectively, which, in our opinion, makes the question indeed intriguing.

\begin{prob}
Find the minimum of the mean widths of $3$-dimensional parallelohedra of unit inradius.
\end{prob}

\textbf{Acknowledgements}\\
The author is indebted to A. Jo\'os and G. Domokos for many fruitful discussions of this problem, and to an unkown referee for helpful suggestions.

\end{document}